\documentclass[12pt,a4paper]{amsart}
\usepackage{a4wide}
\usepackage{amsmath, amssymb, amsfonts,enumerate}
\usepackage[all]{xy}
\usepackage{amscd}
\usepackage{comment}
\usepackage{mathtools}
\usepackage{url}
\usepackage{hyperref}

\newtheorem{theorem}{Theorem}[section]

\newtheorem{corollary}[theorem]{Corollary}
\newtheorem{proposition}[theorem]{Proposition}

 \theoremstyle{definition}
 \newtheorem{definition}[theorem]{Definition}
 \newtheorem{remark}[theorem]{Remark}

 \newtheorem{example}[theorem]{Example}

\numberwithin{equation}{section}

\newcommand {\N}{\mathbb{N}} 
\newcommand {\Z}{\mathbb{Z}} 
\newcommand {\R}{\mathbb{R}} 

\newcommand {\sph}{\mathbb{S}} 

\newcommand{\MM}{\mathcal{M}}

\newcommand{\PP}{\mathcal{P}}

\newcommand{\UU}{\mathcal{U}}

\DeclareMathOperator{\CSA}{CSA}
\DeclareMathOperator{\CMA}{CMA}

\DeclareMathOperator{\End}{End}

\DeclareMathOperator{\Sym}{Sym}
\DeclareMathOperator{\Id}{Id}

\DeclareMathOperator{\Aut}{Aut}
\DeclareMathOperator{\Map}{Map}

\begin{document}
\title[Topological stability]{Topological stability of semigroup actions and shadowing}
\author[T.Ceccherini-Silberstein]{Tullio Ceccherini-Silberstein}
\address{Dipartimento di Ingegneria, Universit\`a del Sannio, I-82100 Benevento, Italy}
\address{Istituto Nazionale di Alta Matematica ``Francesco Severi'', I-00185 Rome, Italy}
\email{tullio.cs@sbai.uniroma1.it}
\author[M.Coornaert]{Michel Coornaert}
\address{Universit\'e de Strasbourg, CNRS, IRMA UMR 7501, F-67000 Strasbourg, France}
\email{michel.coornaert@math.unistra.fr}
\author[X.K.Phung]{Xuan Kien Phung}
\address{D\'epartement d'Informatique et de Recherche Op\'erationnelle, Universit\'e de Montr\'eal, Montr\'eal, Qu\'ebec, H3T 1J4, Canada}
\address{D\'epartement de Math\'ematiques et de Statistique, Universit\'e de Montr\'eal, Montr\'eal, Qu\'ebec, H3T 1J4, Canada}
\email{phungxuankien1@gmail.com}
\begin{abstract}
We investigate expansiveness, topological stability, and shadowing for continuous actions of semigroups on compact Hausdorff spaces.
We characterize semigroups for which all full shifts are expansive.
We show that every expansive continuous monoid action on a compact Hausdorff space which has the shadowing property is topologically stable, and that a subshift with finite alphabet over a monoid has the shadowing property if and only if it is of finite type.
\end{abstract} 
\date{\today}
\maketitle

\setcounter{tocdepth}{1}
\tableofcontents

\section{Introduction}
The notion of \emph{topological stability} for homeomorphisms of compact metric spaces was introduced by Walters in~\cite{walters-anosov}.
Roughly speaking, a homeomorphism $f\colon X\to X$ of a compact metric space $X$ is topologically stable if for every homeomorphism $g\colon X \to X$ sufficiently close to $f$
in the $C^0$-topology,
there exists a continuous map $h \colon X \to X$, that can be taken to be arbitrarily close to the identity map of $X$, such that $h \circ g =f\circ h$.   
Walters proved  that every Anosov diffeomorphism of a compact manifold is topologically stable~\cite[Theorem~1]{walters-anosov}.
A homeomorphism has the \emph{pseudo-orbit tracing property}, or \emph{shadowing property}, when pseudo-orbits stay close to actual orbits
(cf.~\cite[Chapter~2]{aoki-hidaire-book}, \cite[Chapter~5]{brin-stuck}, \cite[Chapter~1]{nekrashevych}). 
The pseudo-orbit tracing property  has a long history in the theory of dynamical systems and  its importance  in the study of axiom A diffeomorphisms was stressed by Bowen \cite{bowen-omega-limit-sets}, ~\cite{bowen-lns-1978}.
Walters investigated the relationships between shadowing and topological stability  in~\cite{walters-potp-1978}.
He proved in particular that every expansive homeomorphism of a compact metric space with the pseudo-orbit tracing property is topologically stable~\cite[Theorem~4]{walters-potp-1978}. 
Recently, Walters' result  was  extended to continuous actions of finitely generated groups
by Chung and Lee~\cite[Theorem~2.8]{chung-lee} and   to continuous actions of countable groups by Meyerovitch~\cite[Theorem~2.4]{meyerovitch-2019}.
\par
In the present paper, we consider dynamical systems of the form $(X,S)$,  where $X$ is a compact Hausdorff space and 
$S$ is a semigroup acting continuously on $X$.
We do not require $X$ to be metrizable nor $S$ to be countable.
For   monoid actions, we shall establish the following extension of  the previously mentioned topological stability results.

\begin{theorem}
\label{t:main}
Let $X$ be a compact Hausdorff space equipped with a continuous action of a monoid $M$.
Suppose that the dynamical system $(X,M)$ is expansive and has the pseudo-orbit tracing property.
Then $(X,M)$ is topologically stable.
\end{theorem}

We shall also obtain the following dynamical characterization of subshifts of finite type over monoids,  which was previously established 
by Walters~\cite[Theorem~1]{walters-potp-1978} (resp.~Oprocha~\cite[Theorem~4.5]{oprocha-2008}, resp.~Chung and Lee~\cite[Theorem~3.2]{chung-lee})
in the  case when $M =\Z$ (resp.~$M = \Z^d$, resp.~$M$ is a finitely generated group)
and $X$ is a compact metric space.

\begin{theorem}
\label{t:sft-potp}
Let $M$ be a monoid and let $A$ be a finite set.
Let $X \subset A^M$ be a subshift.
Then the following conditions are equivalent:
\begin{enumerate}[\rm (a)]
\item
the subshift $X$ is of finite type;
\item
the dynamical system $(X,M)$ has the pseudo-orbit tracing property.
\end{enumerate}
\end{theorem}

From the above results, we deduce the following.

\begin{corollary}
\label{c:sft-top-stab}
Let $M$ be a monoid, let $A$ be a finite set,
and let  $X \subset A^M$ be a subshift of finite type.
Then the dynamical system $(X,M)$ is topologically stable.
\end{corollary}
 
Bucki~\cite{bucki} studied topological stability and shadowing of subshifts over finitely generated free monoids.
Theorem~\ref{t:sft-potp-semigroup} covers \cite[Theorem~3.6]{bucki}, and Theorem~\ref{t:main-semigroup} 
(resp.\ Corollary~\ref{c:sft-w-top-stable}) covers \cite[Theorem~4.8]{bucki} (resp.\ \cite[Corollary~4.9]{bucki}).

\par
For semigroup actions, we show that every expansive dynamical system $(X,S)$
with the pseudo-orbit tracing property is topologically semistable (Theorem~\ref{t:main-semigroup}).
We  extend the dynamical characterization of subshifts of finite type stated for monoids in Theorem~\ref{t:sft-potp} to semigroups $S$ 
admitting a left-identity element (see~Theorem~\ref{t:sft-potp-semigroup}).
\par
The organization of the paper is as follows.
In Section~\ref{sec:background},
we introduce basic notions, fix notation, and establish some preliminary results on expansivity and the pseudo-orbit tracing property.
Section~\ref{sec:top-stab} is devoted to topological semistability and topological stability.
We prove in particular Theorem~\ref{t:main}  and Theorem~\ref{t:main-semigroup}.
Section~\ref{sec:sft} deals with shifts and subshifts over semigroups.
We characterize semigroups whose full shifts are expansive (Proposition~\ref{p:shift-expansive})
and we give the proof of Theorem~\ref{t:sft-potp-semigroup}.
Section~\ref{sec:equicontinuous} is devoted to equicontinuous actions of semigroups.
We prove that every equicontinuous action of a finitely generated monoid on a Stone space has the pseudo-orbit tracing property
(Theorem~\ref{t:equicont-potp}). This extends a result previously obtained by  Chung and Lee~\cite[Theorem~4.1]{chung-lee}.
In Section~\ref{sec:rooted-trees}, we investigate some dynamical properties of the boundary action of endomorphism semigroups of rooted trees.
We observe in particular that the action of the Grigorchuk group on the boundary of its tree has
the pseudo-orbit tracing property but is neither expansive nor topologically stable.
This is also  true for the Basilica group, the Hanoi Towers groups, and, more generally, all weakly branch groups.\\
\par

\noindent
{\bf Acknowledgments.} We thank Dominik Kwietniak for pointing out to our attention Bucki's paper~\cite{bucki} 
which contains results on the topological stability and shadowing of subshifts over finitely generated free monoids. 
Note that Theorem~\ref{t:sft-potp-semigroup} covers \cite[Theorem~3.6]{bucki}, and Theorem~\ref{t:main-semigroup} 
(resp.\ Corollary~\ref{c:sft-w-top-stable}) covers \cite[Theorem~4.8]{bucki} (resp.\ \cite[Corollary~4.9]{bucki}).

\section{Preliminaries}
\label{sec:background}
In this section we fix notation, introduce basic definitions, and establish some preliminary results. 

\subsection{General notation}
We write $\N \coloneqq \{0,1,\dots,\}$ for the set of non-negative integers.
\par
Given a set $X$, we denote by $\PP(X)$ the set of all subsets of $X$.
\par
We use $\sqcup$ to denote disjoint union and $|\cdot|$ for cardinality of finite sets.

\subsection{Semigroup actions}
A \emph{semigroup} is a set equipped with an associative binary operation (see e.g.~\cite{clifford-preston}).
Thus, a semigroup is a set $S$ together with a map $S \times S \to S$, $(s,s') \mapsto ss'$, such that
$(s_1 s_2)s_3 = s_1(s_2s_3)$ for all $s_1,s_2,s_3 \in S$.
\par
Let $S$ be a semigroup.
An element $s \in S$ is called a \emph{left-identity} (resp.~\emph{right-identity}, resp.~\emph{left-zero}, resp.~\emph{right-zero}) element of $S$
if one has $st = t$ (resp.~$ts = t$, resp.~$st = s$, resp.~$ts = s$) for all $t \in S$.
One says that $s \in S$ is an \emph{identity} (resp.~\emph{zero}) element of $S$ if $s$ is both a left-identity and a right-identity (resp.~a left-zero and a right-zero) element of $S$.
If $S$ admits an identity (resp.~a zero) element then this element is unique and usually denoted by $1_S$ (resp.~$0_S$).
One says that a subset $T \subset S$ is a \emph{subsemigroup} of $S$ if one has $ss' \in T$ for all $s,s' \in T$.
Given any subset $\Sigma \subset S$, the intersection of all subsemigroups of $S$ containing $\Sigma$ is the unique minimal subsemigroup of $S$ containing $\Sigma$ and is denoted by $\langle \Sigma \rangle$.
One says that the semigroup $S$ is \emph{finitely generated} if there exists a finite subset $\Sigma \subset S$ such that $S = \langle \Sigma \rangle$.
There is an induced semigroup structure on $\PP(S)$ defined by
$\Sigma \Sigma' \coloneqq \{s s' : s \in \Sigma, s' \in \Sigma'\}$ for all $\Sigma,\Sigma' \in \PP(S)$.
\par
Given two semigroups $S$ and $T$, a \emph{semigroup morphism} from $S$ into $T$ is a map $f \colon S \to T$ such that $f(ss') = f(s)f(s')$ for all $s,s' \in S$.
\par
A \emph{monoid} is a semigroup admitting an identity element.
\par
One says that a subset $N$ of a monoid $M$ is a \emph{submonoid} of $M$ if $N$ is a subsemigroup of $M$ and $1_M \in N$.
\par
If $M$ and $N$ are monoids, a \emph{monoid morphism} from $M$ into $N$ is a semigroup morphism $f \colon M \to N$ such that $f(1_M) = 1_N$.
\par
The class of all semigroups (resp.~monoids) and their morphisms form a category.
The category of groups is a full subcategory of the category of monoids.
The category of monoids is a subcategory of the category of semigroups but it is not a full subcategory.
\par
Given a set $X$, the set $\Map(X)$ of all maps $X  \to X$  is a monoid for the composition of maps.
The identity element of $\Map(X)$ is the identity map $\Id_X \colon X \to X$, $x \mapsto x$.
\par
An \emph{action} of a semigroup  $S$ on a set $X$ is a semigroup morphism $\alpha \colon S \to \Map(X)$.
Writing $\alpha_s \coloneqq \alpha(s)$ for $s \in S$, a map $\alpha \colon S \to \Map(X)$ is a semigroup action if and only if
$\alpha_s \circ \alpha_{s'} = \alpha_{s s'}$ for all $s,s' \in S$.
Equivalently, writing $s x \coloneqq \alpha_s(x)$ for $s \in S$ and $x \in X$,
we may regard  an action of the semigroup $S$ on the set $X$ 
as being defined by a map $S \times X \to X$, $(s,x) \mapsto sx$, that satisfies
$s(s'x) = (ss')x$  for all $s,s' \in S$ and $x \in X$.
One says that an action of a semigroup $S$ on a set $X$ is \emph{transitive} if for all $x,y \in X$ there exists $s \in S$ such that $y = sx$.
\par
If $M$ is a monoid, a \emph{monoid action} of $M$ on $X$ is a monoid morphism $\alpha \colon M \to \Map(X)$.
In other words, a monoid action of $M$ on $X$ is a semigroup action of $M$ on $X$ sending the identity element $1_M$ to the identity map $\Id_X$.
\par
Let $X$ be a topological space and let $C(X)$ denote the submonoid of  $\Map(X)$ consisting of all continuous maps $X \to X$.
One says that an action $\alpha \colon S \to \Map(X)$ of a semigroup  $S$ on $X$ is \emph{continuous} if $\alpha(S) \subset C(X)$.
This amounts to saying that the map $X \to X$, $x \mapsto s x$, is continuous for every $s \in S$.
One says that a continuous action $\alpha \colon S \times X \to X$ is \emph{minimal} if the $\alpha$-orbit set $\{\alpha_s(x):s \in S\}$ of every point $x \in X$ is dense in $X$.
\par
Given continuous actions $\alpha,\beta \colon S \times X \to X$, one 
 says that $\alpha$ is a \emph{factor} of $\beta$ if there exists a continuous surjective map $h \colon X \to X$ such that $\alpha_s \circ h = h \circ \beta_s$ for all $s \in S$.
 One says that $\alpha$ and $\beta$ are \emph{topologically conjugate}
if there exists a homeomorphism $h \colon X \to X$ such that $\alpha_s \circ h = h \circ \beta_s$ for all $s \in S$.
\par
A \emph{dynamical system} is a triple $(X,S,\alpha)$ where $X$ is a compact Hausdorff space, $S$ is a semigroup (resp.~a monoid), and $\alpha \colon S \times X \to X$ is a continuous action of $S$ on $X$.
A dynamical system $(X,S,\alpha)$ will be simply denoted by $(X,S)$ if there is no risk of confusion.

\subsection{Prodiscrete spaces}
Let $(A_i)_{i \in I}$ be a family of finite  sets indexed by a set $I$.
Consider the product set $X \coloneqq \prod_{i \in I} A_i$.
For $i \in I$, we denote by $x(i)$ the projection of $x \in X$ on $A_i$.
Given $x \in X$ and $J \subset I$, we write $x|_J \coloneqq (x_j)_{j \in J} \in \prod_{j \in J} A_j$.
We equip $X$ with its  \emph{prodiscrete topology}, that   is,  the product topology obtained by taking the discrete topology on every factor $A_i$ of $X$.
A base of open sets for the prodiscrete topology on $X$ is formed by the \emph{cylinder sets}
\[
\{x \in X : x|_K = c \},
\]
where $K$ runs over all finite subsets of $I$ and $c$ over all elements of the finite set $\prod_{k \in K} A_k$.
As the product of a family of compact (resp.~Hausdorff, resp.~totally disconnected) spaces is itself compact (resp.~Hausdorff, resp.~totally disconnected), 
the prodiscrete  topology  on $X$ is compact, Hausdorff, and  totally disconnected.
If every $A_i$ has more than one element and  $I$ is countably infinite, then  $X$, with its prodiscrete topology,  is homeomorphic to the triadic Cantor set on the real line and therefore metrizable.
On the other hand, if every $A_i$ has more than one element and $I$ is uncountable, then $X$ is not metrizable for the prodiscrete topology.

\subsection{Stone spaces}
Recall that a  topological space $X$  is said to be \emph{totally disconnected} if the only connected subsets of $X$ are the empty set $\varnothing$ and the singletons $\{x\}$ ($x \in X$).
A \emph{Stone space} is a topological space which is  compact, Hausdorff, and totally disconnected.
It is clear that every topological space which is homeomorphic to a closed subspace of a product of finite discrete sets is a Stone space.
The converse is also true by the Stone representation theorem. 
Actually, a topological space $X$ is a Stone space if and only if there exists a set $I$ such that
$X$ is homeomorphic to a closed subspace of the product space $\prod_{i \in I} A_i$, where $A_i \coloneqq \{0,1\}$ for all $i \in I$
(see for example~\cite[Theorem~2.7]{coo-book}). 
 
   \subsection{The uniform structure of a compact Hausdorff space}
The reader is assumed to have some familiarity with the basic notions in the theory of uniform spaces 
(see e.g.~\cite[Chapitre~II]{bourbaki-top-gen} or \cite[Chapter~6]{kelley}).
\par
Let $X$ be a set. 
We denote by $\Delta_X$ the \emph{diagonal} of $X$, i.e., the subset $\Delta_X \coloneqq \{(x,x):x \in X\} \subset X \times X$.
The \emph{inverse} of a subset $U \subset X \times X$ is the subset
$\widetilde{U} \subset X \times X$ defined by
$\widetilde{U} \coloneqq \{(y,x) : (x,y) \in U\}$.
 A subset $U \subset X \times X$ is called \emph{symmetric} if $\widetilde{U} = U$.
The \emph{composite} of two subsets $U,V \in \PP(X \times X)$ is the subset $U \circ V \in \PP(X \times X)$
consisting of all $(x,y) \in X \times X$ such that there exists $z \in X$ with $(x,z) \in U$ and $(z,y) \in V$.
Observe that $\PP(X \times X)$ is a monoid for $\circ$
whose identity element is $\Delta_X$.
\par
Recall that a  \emph{uniform structure} on a set $X$ is a non-empty subset $\UU \subset \PP(X \times X)$ satisfying the following conditions:
(1)  if $U \in \UU$ then $\Delta_X \subset U$;
(2) if $U \in \UU$ and $U \subset V \subset X \times X$ then $V \in \UU$;
(3) if $U \in \UU$ and $V \in \UU$  then $U \cap V \in \UU$;
(4) if $U \in \UU$ then $\widetilde{U} \in \UU$;
(5)  if $U \in \UU$  then there exists $V \in \UU$ such that $V\circ V \subset U$.
The elements of $\UU$ are then called the \emph{entourages} of $X$ (for the uniform structure $\UU$).
The \emph{topology associated} with a uniform structure $\UU$ on $X$ is the topology on $X$ whose open subsets are the sets $\Omega \subset X$
satisfying the following condition:  for every $x \in \Omega$ there exists $U \in \UU$ such that $\{ y \in X : (x,y) \in U\} \subset \Omega$. 
\par  
If $X$ is a compact Hausdorff space, then its topology  is associated with a unique uniform structure on $X$
(see e.g.~\cite[II, p. 54, Th\'eor\`eme~1]{bourbaki-top-gen}).
The entourages of $X$ for this canonical uniform structure are the neighborhoods in $X \times X$ of
the diagonal $\Delta_X \subset X \times X$.
Note that the open  entourages form a base of the entourages of the uniform structure on $X$ by this definition.
The closed entourages of $X$ form also  a base of entourages of the uniform structure
 (this easily follows from the fact that $X \times X$ is normal since it is compact Hausdorff). 

\begin{remark}
\label{r:entourages-met}
Suppose that  $X$ is a compact metrizable space  and let   $d$  be a metric on $X$ compatible with the topology.
Then the sets 
 $\{(x,y) \in X \times X : d(x,y) \leq  \varepsilon\}$, where $\varepsilon$ runs over all positive real numbers,
 form  a base of entourages for the canonical uniform structure on $X$.
 \end{remark}

\begin{remark}
\label{r:base-entour-prodiscrete}
Let $X$ be a Stone space. 
Let $(A_i)_{i \in I}$ be a family of finite discrete spaces
such that $X$ is homeomorphic to a closed subset 
of $\prod_{i \in I} A_i$ for the prodiscrete topology. 
Then the sets $\{(x,y) \in X \times X : x|_K = y|_K\}$, where $K$ runs over all finite subsets of $I$,
form  a base of entourages for the canonical uniform structure on $X$.
\end{remark}

\subsection{Expansivity}
\label{ss:exp}

\begin{definition}[Expansivity]
\label{def:expansive}
Let $X$ be a compact Hausdorff space equipped with a continuous action of a semigroup $S$. 
One says that an entourage $E$ of $X$ is an \emph{expansivity entourage} for the dynamical system $(X,S)$, or that $(X,S)$ is $E$-\emph{expansive},  if, for all distinct $x,y \in X$, 
there exists $s \in S$ such that $(sx,sy) \notin E$.
\par  
One says that the dynamical system $(X,S)$ is \emph{expansive} if there exists an entourage $E$ of $X$ such that $(X,S)$ is $E$-expansive.
\end{definition}

\begin{remark}
\label{r:subexpentourage}
If  the dynamical system $(X,S)$ is $E$-expansive for some entourage $E$ of $X$,
 then $(X,S)$ is $E'$-expansive for every entourage $E'$ of $X$ such that $E' \subset E$.
\end{remark}

\begin{remark}
\label{rem:subsystem-exp}
If the dynamical system $(X,S)$ is expansive and $Y$ is an $S$-invariant closed subset of $X$, then the dynamical system $(Y,S)$ is also expansive.
Indeed, if $E \subset X \times X$ is an expansivity entourage for $(X,S)$, then $E \cap (Y \times Y)$ is an expansivity entourage for $(Y,S)$.
\end{remark}

\begin{remark}
Suppose that $X$ is a compact metrizable space equipped with a continuous action of a semigroup $S$. 
Let $d$ be a metric on $X$ compatible with the topology.
Then it follows from Remark~\ref{r:entourages-met} 
that the dynamical system $(X,S)$ is expansive if and only if there exists a constant $C > 0$ such that, for all distinct $x,y \in X$, there exists $s \in S$ such that $d(sx,sy) \geq C$.
Such a constant $C$ is then called an \emph{expansivity constant} for
$(X,S,d)$.
\end{remark}

The following result is Lemma~1.2.7 in~\cite{nekrashevych}. 
It extends  Lemma~2 in~\cite{walters-potp-1978} and Lemma~2.10 in~\cite{chung-lee}. 

\begin{proposition}[Uniform expansivity]
\label{p:uniform-exp}
Let $X$ be a compact Hausdorff space equipped with an expansive  continuous action of a semigroup $S$.
Suppose that  $E$ is an expansivity entourage for $(X,S)$ and let $U$ be an entourage of $X$. 
Then there exists a  finite subset $K \subset S$ such that the following holds: 
if $x, y \in X$ satisfy  $(sx,sy) \in E$ for all $s \in K$, 
then $(x,y) \in U$.
\end{proposition}

\begin{proof}
Since the closed (resp.~open) entourages of $X$ form a base of entourages for the uniform structure on $X$ and every entourage contained in an expansivity entourage for $(X,S)$ is also an expansivity entourage for $(X,S)$,
we may assume that $E$ is closed and $U$ is open in $X \times X$.
Suppose by contradiction that,   
for every finite subset $K \subset S$, there exist $x_K, y_K \in X$ such that
$(sx_K,sy_K) \in E$ for all $s \in K$ and $(x_K,y_K) \notin U$.
Consider the directed set $I$ consisting of all finite subsets of $S$ (partially ordered by inclusion)  
and the net $(x_K,y_K)_{K \in I}$ in $X \times X$. 
By compactness of  $X \times X$, the net $(x_K,y_K)_{K \in I}$ admits a cluster point $(x,y) \in X \times X$. 
Since $U$ is open in $X \times X$, we have $(x,y) \notin U$ and hence $x \not= y$.
On the other hand, since $E$ is closed and the action of $S$ on $X$ is continuous, we have $(s x,s y) \in E$ for all $s \in S$.
This contradicts the fact that  $E$ is an expansivity entourage. 
\end{proof}

\subsection{The pseudo-orbit tracing property}\label{ss:POTP}

\begin{definition}[Tracing]
\label{def:trace}
Let $X$ be a compact Hausdorff space equipped with a continuous action of a semigroup $S$.
Given an entourage $U$ of $X$, one says that a family $(x_s)_{s \in S}$ of points of $X$ is $U$-\emph{traced} by the orbit of a point $x \in X$  if one has
$(sx,x_s) \in U$ for all $s \in S$.
\end{definition}

\begin{definition}[Pseudo-orbit]
\label{def:pseudo-orbit}
Let $X$ be a compact Hausdorff space equipped with a continuous action of a semigroup $S$.
Given a finite subset $K \subset S$  and an entourage $V$ of $X$,
one says that a family $(x_s)_{s\in S}$  of points of  $X$ 
is a $(K,V)$-\emph{pseudo-orbit} of $X$ if one has 
$(k x_s, x_{k s}) \in V$
for all $k \in K$ and $s \in S$. 
\end{definition}

\begin{definition}[Pseudo-orbit tracing property]
Let $X$ be a compact Hausdorff space equipped with a continuous action of a semigroup $S$. 
One says that the dynamical system $(X,S)$ has the \emph{pseudo-orbit tracing property} if,
for every entourage $U$ of $X$,  there exist a finite subset $K \subset S$ and an entourage  $V$ of $X$ such that
any $(K,V)$-pseudo-orbit  in $X$ is $U$-traced by the orbit of some
point of $X$.
\end{definition}

\begin{remark}
Suppose that $X$ is a compact metrizable space equipped with a continuous action of a semigroup $S$. 
Let $d$ be a metric on $X$ compatible with the topology.
Given a finite subset $K \subset S$ and $\delta > 0$, one says that a family $(x_s)_{s \in S}$ of points in $X$
is a $(K,\delta)$-\emph{pseudo-orbit} in $X$ if
one has $d(kx_s, x_{ks}) \leq  \delta$
for all $k \in K$ and $s \in S$.
It follows from Remark~\ref{r:entourages-met} that the dynamical system $(X,S)$ has the pseudo-orbit tracing property if and only if
for every $\varepsilon > 0$ there exist a finite subset $K \subset S$ and  $\delta > 0$
such that for every $(K,\delta)$-pseudo-orbit $(x_s)_{s\in S}$ in $X$ there exists
$x \in X$ satisfying $d(sx, x_s) \leq  \varepsilon$ for all $s \in S$
(one then says that the pseudo-orbit $(x_s)_{s \in S}$ is $\varepsilon$-\emph{traced} by the orbit of  $x$).
    \end{remark}

The following result extends  Lemma~3 in~\cite{walters-potp-1978} and Lemma~2.9 in~\cite{chung-lee}. 

\begin{proposition}
\label{p:exp-unique-tracing-point}
Let $X$ be a compact Hausdorff space equipped with a continuous action of a semigroup $S$.
Suppose that $(X,S)$ is expansive and let $E$ be an expansivity entourage for $(X,S)$. 
Let $V$ be a symmetric  entourage of $X$ such that $V \circ V \subset E$.
Then, if $(x_s)_{s \in S}$ is a family of points of $X$, there is at most one point $x \in X$ whose orbit $V$-traces the family $(x_s)_{s \in S}$. 
\end{proposition}

\begin{proof}
Suppose that the family $(x_s)_{s \in S}$ is $V$-traced by the orbits of the points  $x$ and  $y$ ($x_s,x,y \in X$ for all $s \in S$).
This means that  $(sx,x_s),(sy,x_s) \in V$ for all $s \in S$.
As $V$ is symmetric, we deduce that
$(sx,sy) \in V \circ V \subset E$ for all $s \in S$.
This implies $x = y$ since $E$ is an expansivity entourage for $(X,S)$.
\end{proof}

The following result shows in particular that our definition of the pseudo-orbit tracing property  is equivalent to the one given 
in~\cite{osipov-2013} and  \cite{chung-lee} for   finitely generated groups. 

\begin{proposition}
\label{p:potp-fg}
Let $X$ be a compact Hausdorff space equipped with a continuous action of a semigroup $S$. 
Suppose that $S$ is finitely generated and let $\Sigma \subset S$ be a finite generating subset of $S$.
Then the following conditions are equivalent:
\begin{enumerate}[\rm (a)]
\item
the dynamical system $(X,S)$ has the pseudo-orbit tracing property;
\item
for every entourage $U$ of $X$, there exists an entourage $V$ of $X$ such  that any
$(\Sigma,V)$-pseudo-orbit in $X$ is $U$-traced by the orbit of some point of $X$.
\end{enumerate}
\end{proposition}

\begin{proof}
The implication (b) $\implies$ (a) is obvious.
\par
Conversely, suppose (a).
Let $U$ be an entourage of $X$.
Since $(X,S)$ has the pseudo-orbit tracing property, there exist a finite subset $K \subset S$ and an entourage $W$ of $X$ such that any
$(K,W)$-pseudo-orbit in $X$ is $U$-traced by the orbit of some point of $X$.
As $\Sigma$ generates $S$, there exists an integer $n \geq 1$ such that
\begin{equation}
\label{e:K-subset-B-Sigma-n} 
K \subset \Sigma \cup \Sigma^2 \cup \cdots \cup \Sigma^n.
\end{equation}
By compactness of $X$, the action of $S$ on $X$ is uniformly continuous, i.e., the map $X \to X$, $x \mapsto sx$ is uniformly continuous for every $s \in S$.
As $\Sigma$ is finite, we deduce that there exists a sequence $(V_i)_{1 \leq i \leq n}$ of entourages of $X$  satisfying the following conditions:
\begin{enumerate}[\rm (C1)]
\item
$V_i \subset V_{i + 1}$ for all $i \in \{1,2,\dots,n-1\}$;
\item
$(\sigma x,\sigma y) \in V_{i + 1}$ for all $i \in \{1,2,\dots,n-1\}$, $(x,y) \in V_i \circ V_i$, and $\sigma \in \Sigma$;
\item
$V_n \circ V_n  \subset W$.
\end{enumerate}
Let $V \coloneqq V_1$ and suppose that $(x_s)_{s \in S}$ is a $(\Sigma,V)$-pseudo-orbit in $X$.
This means that
\begin{equation}
\label{e:x-s-sigma-n-po}
(\sigma  x_s, x_{\sigma  s}) \in V \quad \text{for all } \sigma \in \Sigma \text{ and } s \in S.
\end{equation}
We claim that the family $(x_s)_{s \in S}$ is a $(K,W)$-pseudo-orbit.
Let  $k \in K$ and $s \in S$.
By~\eqref{e:K-subset-B-Sigma-n}, there exists   $i \in \{1,2,\dots,n\}$ and $\sigma_1,\sigma_2,\dots,\sigma_i \in \Sigma$ such that
$k = \sigma_i  \cdots \sigma_2 \sigma_1$.
Let us show by induction on $i$  that
\begin{equation}
\label{e:x-s-also-K-W}
(k x_s,x_{k s}) \in V_i \circ V_i.
\end{equation}
For $i = 1$, we have $k = \sigma_1$ and~\eqref{e:x-s-also-K-W} follows from~\eqref{e:x-s-sigma-n-po}
since $V = V_1 \subset V_1 \circ V_1$.
Suppose by induction that~\eqref{e:x-s-also-K-W} is true for some $i \in \{1,2,\dots,n-1\}$ and let $\sigma_{i+1} \in \Sigma$.
From~\eqref{e:x-s-also-K-W} and Condition~(C2),
we get
\[
(\sigma_{i+1} \sigma_i \cdots \sigma_2\sigma_1 x_s,\sigma_{i + 1} x_{\sigma_i \cdots \sigma_2\sigma_1 s}) \in V_{i + 1}.
\]
As $(\sigma_{i + 1}x_{\sigma_i \cdots \sigma_2\sigma_1 s},x_{\sigma_{i + 1} \sigma_i \cdots \sigma_2\sigma_1 s}) \in V = V_1 \subset V_{i+1}$ by~\eqref{e:x-s-sigma-n-po} and (C1), it follows that
\[
(\sigma_{i+1} \sigma_i \cdots \sigma_2\sigma_1 x_s,x_{\sigma_{i + 1} \sigma_i \cdots \sigma_2\sigma_1 s}) \in V_{i+1} \circ V_{i+1}.
\]
This completes induction and hence the proof of our claim.
Since $V_i \circ V_i \subset V_n \circ V_n \subset W$ by (C1) and (C3), we deduce from~\eqref{e:x-s-also-K-W} that
 $(x_s)_{s \in S}$ is a $(K,W)$-pseudo-orbit in $X$.
By our choice of $(K,W)$, we conclude that the family $(x_s)_{s \in S}$ is $U$-traced by the orbit of some point of $X$.
This shows that (a) implies (b).
\end{proof}

\section{Topological stability}
\label{sec:top-stab}

Let $X$ be a compact Hausdorff space and let $S$ be a semigroup.
Denote by   $\CSA_S(X)$ the set consisting of all continuous semigroup actions of  $S$ on $X$.
If  $C(X)$ is the set of all continuous maps from $X$ into itself, we have a natural inclusion
\[
\CSA_S(X) \subset  C(X)^S
\]
obtained by sending every $\alpha \in \CSA_S(X)$ to the family $(\alpha_s)_{s \in S}$.
\par
We equip $C(X)$ with the uniform structure of uniform convergence.
A base of entourages of $C(X)$ is formed by the sets
\[
\{(f,g) \in C(X) \times C(X) : (f(x),g(x)) \in U \text{  for all } x \in X\},
\]
where $U$ runs over all entourages of $X$.
The associated topology on $C(X)$ is the topology of \emph{uniform convergence}, or \emph{$C^0$-topology}. 
It is Hausdorff but not compact in general.
Equip   $C(X)^S = \prod_{s \in S} C(X)$ with the product uniform structure.
Observe that  $\CSA_S(X)$ is a closed subset of $C(X)^S$.
\par
A base of entourages for the uniform structure  on $\CSA_S(X)$ is formed by the sets
\[
\{(\alpha,\beta) \in \CSA_S(X) \times \CSA_S(X)  : (\alpha_k(x),\beta_k(x)) \in U \text{ for all } k \in K \text{ and } x \in X\}, 
\]
where $K$ runs over all finite subsets of $S$ and $U$  over all entourages of $X$.

\begin{remark}
\label{r:unif-structure-unif-conver-met}
Suppose that  $X$ is a compact metrizable space  and let   $d$  be a metric on $X$ compatible with the topology.
Then the uniform structure on $C(X)$ is metrizable.
A metric on $C(X)$  compatible with the uniform structure
is the metric $\delta \colon C(X) \times C(X) \to \R$ given by
\[
\delta(f,g) \coloneqq \sup_{x \in X} d(f(x),g(x)) \quad \text{for all } f,g \in C(X).
\]
A base of entourages for the uniform structure on $\CSA_S(X)$ is  formed by the sets
\[
\{(\alpha,\beta) \in \CSA_S(X) \times \CSA_S(X)  : \delta(\alpha_k,\beta_k) \leq  \varepsilon \text{ for all } k \in K\},
\]
where $K$ runs over all finite subsets of $S$ and $\varepsilon$ over all positive real numbers.
\end{remark}

\begin{definition}[Topological semistability]
\label{def:wts}
Let $X$ be a compact Hausdorff space and let $S$ be a semigroup.
Let $\alpha \in \CSA_S(X)$ be a continuous semigroup action of $S$ on $X$.
One says that the dynamical system $(X,S,\alpha)$ is \emph{topologically semistable} if 
there exists a neighborhood $N$ of $\alpha$ in $\CSA_S(X)$ such that
for every $\beta \in N$ 
there exists a continuous map $h \colon X \to X$  satisfying
$\alpha_s \circ h = h \circ \beta_s$ for all $s \in S$.  
\end{definition}

Let $X$ be a compact Hausdorff space and let $M$ be a monoid.
Denote by   $\CMA_M(X)$ the set consisting of all continuous monoid actions of  $M$ on $X$.
Clearly $\CMA_M(X)$ is a closed subset of $\CSA_M(X)$. 

\begin{definition}[Topological stability]
\label{def:top-sta}
Let $X$ be a compact Hausdorff space and let $M$ be a monoid.
Let $\alpha \in \CMA_M(X)$ be a continuous monoid action of $M$ on $X$.
One says that the dynamical system $(X,M,\alpha)$ is \emph{topologically stable} if for every entourage $U$ of $X$
there exists a neighborhood $N$ of $\alpha$ in $\CMA_M(X)$ such that
for every $\beta \in N$ 
there exists a continuous map $h \colon X \to X$ satisfying 
$\alpha_m \circ h = h \circ \beta_m$ for all $m \in M$
and $(h(x),x) \in U$ for all $x \in X$.  
\end{definition}

\begin{theorem}
\label{t:main-semigroup}
Let $X$ be a compact Hausdorff space equipped with a continuous action $\alpha$ of a semigroup $S$.
Suppose that the dynamical system $(X,S,\alpha)$ is expansive and has the pseudo-orbit tracing property.
Then $(X,S,\alpha)$ is topologically semistable.
Moreover, if $E$ is an expansivity entourage for $(X,S,\alpha)$, then the following holds:
there exists a neighborhood $N$ of $\alpha$ in $\CSA_S(X)$ such that if $\beta \in N$ is $E$-expansive then there exists
an injective continuous map $h \colon X \to X$ satisfying $\alpha_s \circ h = h \circ \beta_s$ for all $s \in S$.
\end{theorem}

\begin{proof}
Let $E$ be an expansivity entourage for $(X,S,\alpha)$
and let $V$ be a symmetric entourage of $X$ such that $V \circ V \circ V \subset E$
(the existence of $V$ follows from the properties of a uniform structure).
Since $(X,S,\alpha)$ has the pseudo-orbit tracing property, there exist a finite subset $K \subset S$ and an entourage $W$ of $X$
such that every $\alpha$-$(K,W)$-pseudo-orbit in $X$ is $V$-traced by some $\alpha$-orbit in $X$.
Consider  the neighborhood $N$ of $\alpha \in \CSA_S(X)$  defined by
\[
N \coloneqq \{\gamma \in \CSA_S(X) : (\alpha_k(x),\gamma_k(x)) \in W \text{ for all } k \in K \text{ and } x \in X\}.
\]
Let $\beta \in N$.
Given  $x \in X$, we have,  for all $k \in K$ and $s \in S$,
\begin{equation*}
(\alpha_k(\beta_s(x)),\beta_{ks}(x)) = (\alpha_k(\beta_s(x)),\beta_k(\beta_s(x))) \in W.
\end{equation*} 
 Therefore, the $\beta$-orbit  $(\beta_s(x))_{s \in S}$  of $x$ is an $\alpha$-$(K,W)$-pseudo-orbit in $X$.
By the choice of $K$ and $W$, it follows that the $\beta$-orbit of $x$ is $V$-traced by the $\alpha$-orbit of some point of $X$, 
i.e., there exists  $y \in X$ such that
\begin{equation}
\label{e:proof-wts-1} 
(\alpha_s(y),\beta_s(x)) \in V \quad \text{for all }s \in S.
\end{equation}
Since $V$ is a symmetric entourage of $X$ and $V \circ V \subset V \circ V \circ V \subset E$,
we deduce from Proposition~\ref{p:exp-unique-tracing-point} that  such a point $y$ is uniquely determined by $x$. 
Consider the map $h \colon X \to X$ defined by $h(x) \coloneqq y$ for all $x \in X$.
\par
Let us first prove  that
\begin{equation}
\label{e:f-semi-conj}  
\alpha_s \circ h = h \circ \beta_s  \quad \text{for all }s \in S.
\end{equation}
Using~\eqref{e:proof-wts-1}, we get, for all $s,t \in S$,
\[
(\alpha_t(\alpha_s(h(x))),\beta_t(\beta_s(x))) = (\alpha_{ts}(h(x)),\beta_{ts}(x)) \in V.
\]
Thus, the $\beta$-orbit of $\beta_s(x)$ is $V$-traced by  the $\alpha$-orbit of $\alpha_s(h(x))$. 
By definition of $h$, this implies that 
$\alpha_s(h(x)) = h(\beta_s(x))$ for all  $s \in S$  and  $x \in X$,
and~\eqref{e:f-semi-conj} follows. 
\par
Let us show now that the map $h \colon X \to X$ is continuous.
Let $U_0$ be an entourage of $X$.
By Proposition~\ref{p:uniform-exp}, there exists a finite subset $K_0 \subset S$  such that
if $y_1,y_2 \in X$ satisfy $(\alpha_s(y_1),\alpha_s(y_2)) \in E$ for all $s \in K_0$, then $(y_1,y_2) \in U_0$.
For every $s \in S$, the map $\beta_s \colon X \to X$ is continuous and therefore uniformly continuous by compactness of $X$.
Since $K_0$ is finite,
it follows that there exists an entourage  $V_0$ of $X$ such that if $x_1,x_2 \in X$ satisfy $(x_1,x_2) \in V_0$ then $(\beta_s(x_1),\beta_s(x_2)) \in V$ for all $s \in K_0$.
As $(\alpha_s(h(x_1)),\beta_s(x_1)),(\beta_s(x_2),\alpha_s(h(x_2))) \in V$ by construction of $h$ and the symmetry of $V$,
we deduce that
\[
(\alpha_s(h(x_1)),\alpha_s(h(x_2))) \in V \circ V \circ V \subset E
\]
for all $x_1,x_2 \in X$ such that $(x_1,x_2) \in V_0$ and $s \in K_0$.
By our choice of $K_0$, we conclude that $(h(x_1),h(x_2)) \in U_0$ for all $x_1,x_2 \in X$ such that $(x_1,x_2) \in V_0$.
This shows that  $h$ is  continuous 
and completes the proof of the topological semistability of the dynamical system  $(X,S,\alpha)$.
\par
Suppose now that $\beta \in N$ is $E$-expansive.
Let $x,x' \in X$ such that $h(x) = h(x')$.
We then deduce from~\eqref{e:proof-wts-1} and the symmetry of $V$ that
$(\beta_s(x),\beta_s(x')) \in V \circ V $ for all $s \in S$.
As $V \circ V \subset E$ and $\beta$ is $E$-expansive, this implies $x = x'$.
Therefore $h$ is injective.
\end{proof}

\begin{corollary}
\label{c:minimal-factor}
Let $X$ be a compact Hausdorff space equipped with a continuous semigroup action $\alpha$ of a semigroup $S$.
Suppose that the dynamical system $(X,S,\alpha)$ is expansive, minimal,  and has the pseudo-orbit tracing property.
Then there exists a neighborhood $N$ of $\alpha$ in $\CSA_S(X)$ such that $\alpha$ is a factor of every $\beta \in N$.
\end{corollary}

\begin{proof}
By Theorem~\ref{t:main-semigroup}, there is a neighborhood $N$ of $\alpha$ in $\CSA_S(X)$ such that, for every $\beta \in N$, there is a continuous map $h \colon X \to X$ such that
$\alpha_s \circ h = h \circ \beta_s$ for all $s \in S$.
Given an arbitrary point $x \in X$, we have
$h(\beta_s(x)) = \alpha_s(h(x))$ for all $s \in S$. 
This implies that the set
$\{h(\beta_s(x)) : s \in S\}$ is dense in $X$ by minimality of $\alpha$.
As $h(X)$ is closed in $X$ by compactness of $X$,
we deduce that $h(X) = X$.
This shows that $h$ is surjective and hence that $\alpha$ is a factor of $\beta$.
\end{proof}

A topological space is said to be \emph{incompressible} if it is not homeomorphic to any of its proper subspaces~\cite{fletcher-sawyer}.
Every closed finite-dimensional topological manifold is incompressible by Brouwer's invariance of domain.

\begin{corollary}
\label{c:top-conj}
Let $X$ be a  compact Hausdorff space  equipped with a continuous semigroup action $\alpha$ of a semigroup $S$.
Suppose that the dynamical system $(X,S,\alpha)$ is expansive and has the pseudo-orbit tracing property.
Suppose also that  the space $X$ is incompressible (e.g., $X$ is a closed finite-dimensional topological manifold)
or that the action $\alpha$ is minimal.
Let $E$ be an expansivity entourage for $(X,S,\alpha)$.
Then there exists a neighborhood $N$ of $\alpha$ in $\CSA_S(X)$ such that 
every $E$-expansive action   $\beta \in N$ is 
topologically conjugate to $\alpha$.
\end{corollary}

\begin{proof}
By~Theorem~\ref{t:main-semigroup},
there exists a neighborhood $N$ of $\alpha$ in $\CSA_S(X)$ such that 
for every $E$-expansive action   $\beta \in N$,
there is an injective continuous map $h \colon X \to X$ such that $\alpha_s \circ h = h \circ \beta_s$ for all $s \in S$.
If $X$ is incompressible, then  $h$ is a homeomorphism since $X$ is a compact Hausdorff space.
On the other hand, if $\alpha$ is minimal, it follows from the proof of Corollary~\ref{c:minimal-factor} that $h$ is surjective, and hence a homeomorphism of $X$. 
\end{proof}

\begin{proof}[Proof of Theorem~\ref{t:main}]
Suppose that  $U$ is an entourage of $X$.
Let us return to the   proof of Theorem~\ref{t:main-semigroup} above.
After replacing $E$ by $E \cap U$ if necessary, we may assume that $E \subset U$.
Consider the neighborhood $N_1$ of $\alpha$ in $\CMA_M(X)$ defined by $N_1 \coloneqq N \cap \CMA_M(X)$.
By taking $s \coloneqq 1_M$ in~\eqref{e:proof-wts-1}, we get,
for any $\beta \in N_1$,
\[ 
(h(x),x) = (y,x) = (\alpha_{1_M}(y),\beta_{1_M}(x)) \in V \subset V \circ V \circ V \subset  E \subset U 
\]
for all $x \in X$.
This shows that the dynamical system $(X,M,\alpha)$ is topologically stable.
\end{proof}

\section{Dynamical characterization of subshifts of finite type}
\label{sec:sft}

\subsection{Shifts and subshifts}
\label{ss:shifts}
Let $S$ be a semigroup  and let $A$ be a finite set.
The elements of $A^S$ are called the \emph{configurations} over the semigroup $S$ and the \emph{alphabet} $A$.
\par
We equip $A^S = \prod_{s \in S} A$ with its prodiscrete topology. 
We also equip $A^S$ with the action of $S$ defined by $(s,x) \mapsto sx  \coloneqq  x \circ R_s$
for all $s \in S$ and $x \in A^S$, where $R_s \colon S \to S$ denotes the right multiplication by $s$.
Thus, the configuration $s x$ is given by $sx(s') = x(s' s)$ for all  $s' \in S$.
This action of $S$ on $A^S$ is clearly  continuous. It is called the $S$-\emph{shift}, or simply the \emph{shift}, on $A^S$.

\begin{remark}
In the case when $M$ is a monoid, the shift action of $M$ on $A^M$ is a monoid action since
$1_M x = x \circ R_{1_M} = x \circ \Id_M = x$ for all $x \in A^M$.
\end{remark}

A closed $S$-invariant subset $X \subset A^S$ is called a \emph{subshift}.

\begin{proposition}
\label{p:shift-expansive}
Let $S$ be a semigroup and let $A$ be a finite set with more than one element.
Then the following conditions are equivalent:
\begin{enumerate}[\rm (a)]
\item
the dynamical system $(A^S,S)$ is expansive;
\item
there exists a finite subset $K \subset S$ such that $S = KS$.
\end{enumerate}
\end{proposition}

\begin{proof}
Suppose (a). Let $E$ be an expansivity entourage for $(A^S,S)$.
By Remark~\eqref{r:base-entour-prodiscrete}, there exists a finite subset $K \subset S$ such that
\begin{equation}
\label{e:K-cylinder-U}
\{(x,y) \in A^S \times A^S : x|_K = y|_K\} \subset E.
\end{equation}
Let $s \in S$.
As the set $A$ has more than one element, there exist $x,y \in A^S$ such that $x(s) \not= y(s)$ and $x(t) = y(t)$ for all $t \in S \setminus \{s\}$.
By the $E$-expansivity of  $(A^S,S)$, there exists $s' \in S$ such that $(s'x,s'y) \notin E$.
We then have    $(s'x)|_K \not= (s'y)|_K$ by~\eqref{e:K-cylinder-U}.
Thus, there exists $k \in K$ such that $(s'x)(k) \not= (s'y)(k)$, i.e., $x(ks') \not= y(ks')$.
By our choice of $x$ and $y$, this implies that $s = ks' $. 
As $s \in S$ was arbitrary, we
deduce that $S = KS$. This shows that (a) implies (b).
\par
Conversely, suppose (b).
Consider the entourage $E$ of $A^S$ defined by
\begin{equation*}
E \coloneqq \{(x,y) \in A^S \times A^S : x|_K = y|_K\}.
\end{equation*}
Let $x,y \in A^S$ such that $x \not= y$.
This means that  there exists $s \in S$ such that $x(s) \not= y(s)$.
   By (b), there exist $k \in K$ and $s' \in S$ such that $s = ks'$.
   We then have $(s'x)(k) = x(ks') = x(s) \not= y(s) = y(ks') = (s'y)(k)$.
   This implies $(s'x,s'y) \notin E$.
   Thus, $E$ is an expansivity entourage for $(A^S,S)$.
   This shows that (b) implies (a).
    \end{proof}

\begin{remark}
\label{rem:monoid-exp}
Condition~(b) is satisfied if $S$ admits a left-identity element $e$ by taking
$K \coloneqq \{e\}$.
In particular, all monoids satisfy (b).
\end{remark}

\begin{remark}
A \emph{right-zero-semigroup} is a semigroup $S$ such that $st = t$ for all $s,t \in S$. 
This amounts to saying  that every element in $S$ is a left-identity element
or that every element in $S$ is a right-zero element.
It follows from Remark~\ref{rem:monoid-exp} that all right-zero-semigroups satisfy Condition~(b): one can then take $K \coloneqq \{s\}$ for any $s \in S$.
\end{remark}

\begin{remark}
A \emph{regular semigroup} is a semigroup $S$ such that for every element $s \in S$ there exists an element $t \in S$ such that
$s = sts$. It is clear that every finitely generated regular semigroup $S$ satisfies Condition~(b): one can then take $K$ to be any finite generating subset for $S$. We thank Miguel Donoso Echenique for pointing this out to us.
\end{remark}

\begin{remark}
The direct product of two semigroups satisfying Condition~(b) also satisfies Condition~(b).
Indeed, if $S_1$ and $S_2$ are semigroups and $K_1 \subset S_1$ and $K_2 \subset S_2$  satisfy $K_1S_1 = S_1$ and $K_2S_2 = S_2$,
then $(K_1 \times K_2)(S_1 \times S_2) = K_1S_1 \times K_2S_2 = S_1 \times S_2$.
\end{remark}

\begin{remark}
\label{r:KS=S-sans-left-identity}
Let $S_1, \ldots, S_n$, $n \geq 2$, be mutually disjoint semigroups.
Set $S \coloneqq S_1 \sqcup  \cdots \sqcup S_n \sqcup \{z\}$, where
$z$ is a new element symbol, and define a binary operation $*$ on $S$ by setting,
for all $s,s' \in S$,
\[
s*s' \coloneqq 
\begin{cases}
ss' &\text{ if   there exists } i \in \{1,\dots,n\} \text{ such that }s,s' \in S_i,\\
 z & \text{ otherwise.}
 \end{cases}
 \]
It is straightforward that the operation $*$ is associative and that $z$ is a zero element of $S$.
Suppose that $S_i$ admits a left-identity element $e_i$ for each $i \in \{1, \ldots, n\}$.
Then the semigroup $S$ satisfies Condition~(b) by taking $K \coloneqq \{e_1, \ldots, e_n\}$.
Note that  there is no subset $K' \subset S$ such that $|K'| < |K|$
 and $K'S =S$. 
 This implies in particular that  $S$ does not contain any left-identity element.
\end{remark}

\begin{remark}
Condition (b) is not satisfied if $S S \subsetneqq S$, that is, if the binary operation $S \times S \to S$ is not surjective.
Examples of such semigroups are provided by the additive semigroup of positive integers and all null semigroups with more than one element.
\end{remark}

\subsection{Subshifts of finite type}
Let $S$ be a semigroup  and let $A$ be a finite set.
Given a finite subset $W \subset S$ and $P \subset A^W$, the subset $X \subset A^S$ defined by
\begin{equation}
\label{e:sft}
X \coloneqq \{x \in A^S : (sx)|_W \in P \text{ for all } s \in S\}
\end{equation}
is a subshift. One says that a subshift $X \subset A^S$ is \emph{of finite type} if there exist a finite subset $W \subset S$ and $P \subset A^W$ such that $X$ satisfies~\eqref{e:sft}.
One then says that $W$ is a \emph{defining window} for the subshift of finite type $X$.
Note that if $X \subset A^S$ is a subshift of finite type and $W \subset S$ is a defining window
then~\eqref{e:sft} is satisfied by taking $P \coloneqq X|_W$.   

\begin{proposition}
\label{p:sft-potp-if-KS}
Let $S$ be a semigroup and let $A$ be a finite set.
Suppose that there exists a finite subset $K \subset S$ such that $S = KS$.
Let $X \subset A^S$ be a subshift.
Suppose that  the dynamical system $(X,S)$ has the pseudo-orbit tracing property.
Then the subshift $X$ is of finite type.
\end{proposition}

\begin{proof}
Consider the entourage $U$ of $X$ defined by  
\[
U \coloneqq \{(x,y) \in X \times X : x\vert_K = y\vert_K\}.
\]
Since $(X,S)$ has the pseudo-orbit tracing property,
there exist a finite subset $T \subset S$ and an entourage $V$ of $X$ such that every $(T,V)$-pseudo-orbit in $X$ is $U$-traced by the orbit of some point of $X$.
Let $H \subset S$ be a finite subset such that
\begin{equation}
\label{e:H-sub-V-potp-1}
\{(x,y) \in X \times X : x|_H = y|_H \} \subset V.
\end{equation}
Consider the finite subset $W \coloneqq K \cup H \cup H T \subset S$ and the subshift of finite type
$Z \subset A^S$ (with defining window $W$) defined by
\[
Z \coloneqq \{z \in A^S : (sz)|_W \in X|_W \text{ for all } s \in S\}.
\]
Let us show that $X = Z$.
The inclusion $X \subset Z$ is obvious.
Indeed, every $x \in X$ satisfies $sx \in X$ and hance $(sx)|_W \in X|_W$ for all $s \in S$.
To prove the reverse inclusion, let $z \in Z$.
Then, for every $s \in S$, there exists a configuration $x_s\in X$ such that $(sz)|_W = x_s|_W$.
For all $t \in T$ and $h \in H$, we have
\begin{align*}
(t x_s)(h) &= x_s(h t) \\
&= (s z)(h t) && \text{(since $H T \subset W$)} \\
&= (tsz)(h) \\
&= x_{ts}(h) &&\text{(since $H \subset W$)}.    
\end{align*}
We deduce that $(tx_s)|_H = x_{ts}|_H$ for all $s \in S$ and $t \in T$.
By~\eqref{e:H-sub-V-potp-1}, this implies that the family $(x_s)_{s \in S}$ is a $(T,V)$-pseudo-orbit.
Consequently, there exists a configuration $x \in X$ whose orbit $U$-traces the family $(x_s)_{s \in S}$, i.e., such that
$(sx)(k) = x_s(k)$ for all $s \in S$ and $k \in K$.
As $K \subset W$, we deduce that $z(ks) = (sz)(k) = x_s(k) = (sx)(k) = x(ks)$ for all $k \in K$ and $s \in S$.
Since $KS = S$, this shows that $z = x \in X$.
We conclude  that $X = Z$ and hence that $X$ is of finite type.
\end{proof}

\begin{theorem}
\label{t:sft-potp-semigroup}
Let $S$ be a semigroup and let $A$ be a finite set.
Suppose that the semigroup $S$ admits a left-identity element (e.g., $S$ is a monoid).
Let $X \subset A^S$ be a subshift.
Then the following conditions are equivalent:
\begin{enumerate}[\rm (a)]
\item
the subshift $X$ is of finite type;
\item
the dynamical system $(X,S)$ has the pseudo-orbit tracing property.
\end{enumerate}
\end{theorem}

\begin{proof}
Let $e$ be a left-identity element of $S$.
The implication (b) implies (a) follows from Remark~\ref{rem:monoid-exp} and Proposition~\ref{p:sft-potp-if-KS}.
\par
Conversely, suppose (a) and let $W \subset S$ be a defining window for $X$.
Let $U$ be an entourage of $X$.
By Remark~\ref{r:base-entour-prodiscrete}, there exists a finite subset $H \subset S$ such that
\begin{equation}
\label{e:W-H-subset-U}
V \coloneqq \{(x,y) \in X \times X : x|_H = y|_H \} \subset U.
\end{equation}
Up to replacing $H$ by $H \cup \{e\}$, we may suppose that $e \in H$.
\par
Set $T \coloneqq W \cup H$ and suppose that a family $(x_s)_{s\in S}$ is a $(T,V)$-pseudo-orbit in $X$.
This means that
\begin{equation}
\label{e:xs-pseudo-orbit}
(tx_s)(h) = x_{ts}(h) \text{ for all } t \in T, h \in H, \text{ and } s \in S.
\end{equation}
Define the configuration $x \in A^S$ by setting $x(s) \coloneqq x_{s}(e)$ for all $s \in S$.
For all $t \in T$ and $s \in S$, we then have, on the one hand,
$(sx)(t) = x(ts) = x_{ts}(e)$ and, on the other hand, by using~\eqref{e:xs-pseudo-orbit} and recalling that $e \in H$,
$x_s(t) = x_s(et) = (t x_s)(e) = x_{ts}(e)$.
We deduce that
\begin{equation}
\label{e:sxt}
(sx)(t) = x_s(t) \quad \text{for all } t \in T \mbox{ and } s \in S.
\end{equation}
\par
Since $W \subset T$, it follows from~\eqref{e:sxt} that
\[
(sx)|_W = x_s|_W \in X|_W   \text{  for all } s \in S.
\]
As $W$ is a defining window for $X$, this shows that $x \in X$.
\par
Using now the fact that $H \subset T$, we deduce from~\eqref{e:sxt} that $(sx)|_H = x_s|_H$ for all $s \in S$.
By~\eqref{e:W-H-subset-U}, this implies that $(sx, x_s) \in U$ for all $s \in S$.
Thus, the family $(x_s)_{s \in S}$ is $U$-traced by the orbit of $x$.
This shows that $(X,S)$ has the pseudo-orbit tracing property
and hence that (a) implies (b).
\end{proof}

\begin{corollary}
\label{c:sft-w-top-stable}
Let $S$ be a semigroup admitting a left-identity element  and let $A$ be a finite set.
Suppose that $X \subset A^S$ is a subshift of finite type.
Then the dynamical system $(X,S)$ is  topologically semistable.
\end{corollary}

\begin{proof}
This follows from Theorem~\ref{t:main-semigroup} and implication (a) $\implies$ (b) in Theorem~\ref{t:sft-potp-semigroup}
since $(X,S)$ is expansive by Remark~\ref{rem:monoid-exp}, Proposition~\ref{p:shift-expansive},
and Remark~\ref{rem:subsystem-exp}.
\end{proof}

\begin{proof}[Proof of Corollary~\ref{c:sft-top-stab}]
This follows from Theorem~\ref{t:main} and  Theorem~\ref{t:sft-potp}
since $(X,M)$ is expansive by Proposition~\ref{p:shift-expansive},
Remark~\ref{rem:monoid-exp}, and Remark~\ref{rem:subsystem-exp}.
\end{proof}

\section{Equicontinuous actions and the pseudo-orbit tracing property}
\label{sec:equicontinuous}

\begin{definition}[Equicontinuity]
Let $X$ be a compact Hausdorff space equipped with an  action of a semigroup $S$. 
One says that the action of $S$ on $X$  is \emph{equicontinuous} if for every entourage $U$ of $X$ 
there exists an entourage $V$ of $X$ such that for all $x, y \in X$, $(x, y) \in V$ implies $(sx, sy) \in U$  for all $s \in S$.
\end{definition}

It is clear that every equicontinuous action of a semigroup  on a compact Hausdorff space  is continuous.

\begin{example}[Uniformly Lipschitz actions]
\label{ex:unif-lipshcitz-equi}
Let $(X,d)$ be a compact metric space equipped with an action of a semigroup $S$.
Suppose that  the action of $S$ on $(X,d)$ is $L$-\emph{Lifschitz} for some real number $L > 0$, i.e., 
$d(sx,sy) \leq L d(x,y)$ for all $x,y \in X$ and $s \in S$.
Then   the action of $S$ on $X$  is equicontinuous.
Indeed, given an entourage $U$ of $X$,  there exists $\varepsilon > 0$ such that
$W \coloneqq \{(x,y) \in X \times X : d(x,y) \leq \varepsilon\} \subset U$ 
and the entourage $V \coloneqq \{(x,y) \in X \times X : d(x,y) \leq \varepsilon/L\}$ satisfies
$(sx,sy) \in W \subset U$ for all $(x,y) \in V$ and $s \in S$, showing equicontinuity.
 \end{example}

\begin{example}[Isometric actions]
\label{ex:isom-equi}
Let $(X,d)$ be a compact metric space equipped with an action of a semigroup $S$.
Suppose that  the action of $S$ on $X$ is \emph{isometric}, i.e.,  
$d(sx,sy) =  d(x,y)$ for all $x,y \in X$ and $s \in S$.
Then the action of $S$ on $X$  is $1$-Lifschitz  and hence equicontinuous by 
Example~\ref{ex:unif-lipshcitz-equi}.
\end{example}

\begin{example}[Profinite actions]
\label{ex:profinite-actions}
Let $(A_i)_{i \in I}$ be a family of finite sets indexed by a set $I$
and suppose that a semigroup $S$ acts on the set $A_i$ for every $i \in I$.
Then $S$ acts diagonally on $X \coloneqq \prod_{i \in I} A_i$ by setting $(sx)(i) \coloneqq sx(i)$ for all $s \in S$ and $x \in X$.
This action is equicontinuous with respect to the prodiscrete topology on $X$.
Indeed, if $U$ is an entourage of $X$, it follows from Remark~\ref{r:base-entour-prodiscrete} that there exists a finite subset $K \subset I$ such that
\[
V \coloneqq  \{(x,y) \in X \times X : x|_K = y|_K \} \subset U.
\]
We then have $(sx,sy) \in V \subset U$ for all $(x,y) \in V$ and $s \in S$, showing equicontinuity.
\end{example}

\begin{proposition}
\label{p:not-exp-if-infinite}
Let $X$ be an infinite compact Hausdorff space equipped with an equicontinuous action of a semigroup $S$.
Then the dynamical system $(X,S)$ is not expansive.
\end{proposition}

\begin{proof}
Let $U$ be an entourage of $X$.
By equicontinuity, there exists an entourage $V$ of $X$ such that $(x,y) \in V$ implies $(sx,sy) \in U$ for all $s \in S$.
We have $V \not= \Delta_X$ since otherwise the topology on $X$ would be the discrete one, contradicting the hypothesis that $X$ is infinite compact.
Consequently, there exist distinct points $x,y \in X$ such that  $(sx,sy) \in U$ for all $s \in S$.
This shows that $(X,S)$ is not $U$-expansive.
As the entourage $U$ was arbitrary, we conclude that $(X,S)$ is not expansive.
\end{proof}

The following result extends Theorem~4.1 in~\cite{chung-lee}.

\begin{theorem}
\label{t:equicont-potp}
Suppose that a Stone space $X$ is equipped with an  equicontinuous action of a finitely generated monoid $M$.
Then the dynamical system $(X,M)$ has the pseudo-orbit tracing property.  
\end{theorem}

\begin{proof}
As every Stone space is homeomorphic to a closed subspace of a product of finite discrete spaces,
we can assume that there is a family $(A_i)_{i \in I}$ of finite sets such that $X$ 
is a closed subset of $\prod_{i \in I} A_i$ for the prodiscrete topology. Let $U$ be an entourage of $X$.
By Remark~\ref{r:base-entour-prodiscrete}, there exists a finite subset $K \subset I$ such that
the entourage $W$ of $X$ defined by
$W \coloneqq \{(x,y) \in X \times X : x|_K = y|_K \}$ satisfies
\begin{equation}
\label{e:W-K-subset-U-5}
W \subset U.
 \end{equation}
Since the action of $M$ on $X$ is equicontinuous, there exists an entourage $V$ of $X$ such that
$(x,y) \in V$ implies $(m x,m y) \in W$ for all $x,y \in X$ and $m \in M$.
\par
 Let $\Sigma$ be a finite generating subset of $M$.
\par
Suppose that a family $(x_m)_{m \in M}$ of points of $X$ is a $(\Sigma,V)$-pseudo-orbit in $X$.
This means that, for all $\sigma \in \Sigma$ and $m \in M$,
we have $(\sigma x_m , x_{\sigma m}) \in V$.
By our choice of $V$, this implies $(m'\sigma x_m , m'x_{\sigma m}) \in W$ for all $m'\in M$, that is,
\begin{equation}
\label{e:x-m-po-sigma-m'}
(m'\sigma x_m )|_K =  (m'x_{\sigma m})|_K \quad \text{for all } m,m' \in M \text{ and } \sigma \in \Sigma.
\end{equation}  
Consider now the point $x \coloneqq x_{1_M} \in X$.
We claim that the family $(x_m)_{m \in M}$ is $U$-traced by the orbit of $x$.
To see this, let $m \in M$.
Since $\Sigma$ generates $M$, there exist a integer $n \geq 0$ and $\sigma_1,\dots,\sigma_n \in \Sigma$ such that $m = \sigma_1 \cdots \sigma_n$.
By successive applications of~\eqref{e:x-m-po-sigma-m'}, we get
\begin{align*}
(mx)|_K 
&= (\sigma_1 \cdots \sigma_{n-1} \sigma_n x_{1_M})|_K \\
&= (\sigma_1 \cdots \sigma_{n-1}  x_{\sigma_n})|_K \\
&= (\sigma_1 \cdots \sigma_{n-2}  x_{\sigma_{n_1}\sigma_n})|_K \\
&= \dots \\
 &= x_{\sigma_1 \cdots \sigma_{n-1} \sigma_n}|_K = x_{m}|_K.
\end{align*}
We deduce that $(mx,x_m) \in W$ for all $m \in M$.
 As $W \subset U$ by~\eqref{e:W-K-subset-U-5}, this proves our claim.
It follows that $(X,M)$ has the pseudo-orbit tracing property.
\end{proof}

\begin{remark}
For an arbitrary compact Hausdorff space $X$, Theorem~\ref{t:equicont-potp} becomes false in general.
For example, given any $\theta \in \R$,   the action  $\alpha$ of the additive group $\Z$ on the circle $\sph^1 \coloneqq \R/\Z$  generated by the rotation of angle $\theta$, i.e., 
defined by $\alpha_n(x) \coloneqq x + n\theta \mod 1$ for all $n \in \Z$ and $x \in \sph^1$,
is  equicontinuous but does not have the pseudo-orbit tracing property
(cf.~\cite[Exercise~5.3.4]{brin-stuck}).
\end{remark}

As a product of finite sets is a Stone space for the prodiscrete topology, an immediate consequence of Theorem~\ref{t:equicont-potp} is the following.

\begin{corollary}
\label{c:equicont-potp}
Let $(A_i)_{i \in I}$ be a family of finite sets indexed by a set $I$. 
Suppose that the set  $X \coloneqq \prod_{i \in I} A_i$,  with its prodiscrete topology,
is equipped with an  equicontinuous action of a finitely generated monoid $M$.
Then the dynamical system $(X,M)$ has the pseudo-orbit tracing property.
\qed  
\end{corollary}

\section{Rooted tree endomorphism semigroups}
\label{sec:rooted-trees}

A \emph{rooted tree} is a sequence $T = (V_n,\pi_n)_{n \in \N}$, where $V_0$ is a singleton, $V_{n + 1}$ is a finite set, $|V_{n + 1}| > |V_n|$  and $\pi_n \colon V_{n + 1} \to V_n$ is a surjective  map for every $n \in \N$.
The unique element of  $V_0$ is called the \emph{root} of   $T$ and the elements of the set $V_n$, $n \in \N$,  are called the \emph{vertices of level $n$} of $T$.
A \emph{ray} of the rooted tree $T$ is a sequence $(v_n)_{n \in \N}$ such that $v_n \in V_n$ and $v_n = \pi_n(v_{n + 1})$ for all $n \in \N$.
The \emph{boundary} of the rooted tree $T$ is the subset  $\partial T \subset P \coloneqq \prod_{n \in \N} V_n$  consisting of all the rays of $T$.
We equip $\partial T$ with the topology induced by the prodiscrete topology on $P$.
It is clear that $\partial T$ is an infinite  closed subset of $P$.
Therefore $\partial T$ is an infinite  Stone space.
In fact, $\partial T$ is the limit of the inverse sequence of finite discrete spaces $(V_n,\pi_n)_{n \in \N}$.
The  metric $d$ on $\partial T$ defined, for all rays $\xi,\eta \in \partial T$,  by
setting 
\begin{equation}
\label{e:distance-frontiere}
d(\xi,\eta) \coloneqq
\begin{cases}
0  & \text{ if } \xi = \eta, \\
2^{- \inf\{n \in \N: \xi(n) \not= \eta(n)\}} & \text{ if } \xi \not= \eta,
\end{cases}
\end{equation}
 is compatible with the topology.
 As $\partial T$ is a non-empty metrizable Stone space without isolated point, it is homeomorphic to the triadic Cantor set. 
 \begin{example}
Let $k \geq 2$ be  an integer and set $A \coloneqq \{0,1,\dots,k-1\}$. 
The \emph{$k$-regular rooted tree} is the rooted tree $T_k \coloneqq  (A^n,\pi_n)_{n \in \N}$, 
where $\pi_n  \colon A^{n +1} = A^n \times A \to A^n$ is 
the projection map.
Observe that $\partial T_k$ can be identified with $A^{\N}$
via the map that sends every ray $(v_n)_{n \in \N}$
to the map $\rho \colon \N \to A$, where $\rho(n)$ equals the last coordinate of $v_{n + 1} \in A^{n+1}$ for every $n \in \N$.
With this identification, the topology on $\partial T_k = A^{\N}$ is the prodiscrete topology. 
 \end{example}
 
 Let $T = (V_n,\pi_n)_{n \in \N}$ be a rooted tree.
 One says that a sequence  $f = (f_n)_{n \in \N}$, where $f_n \in \Map(V_n)$ for every $n \in \N$,  is an \emph{endomorphism} of the rooted tree $T$ if
one has  $\pi_n \circ f_{n + 1} = f_n \circ \pi_n$ for all $n \in \N$.
 Every endomorphism $f = (f_n)_{n \in \N}$ of $T$ induces a continuous map $\partial f \colon \partial T \to \partial T$ defined by $\partial f(\xi) \coloneqq (f_n(\xi(n))_{n \in \N}$ for all
 $\xi \in \partial T$.
 The set $\End(T)$ consisting of all endomorphisms of $T$ is a monoid for the binary operation defined by
 $fg \coloneqq (f_n \circ g_n)_{n \in \N}$ for all $f= (f_n)_{n \in \N}, g = (g_n)_{n \in \N} \in \End(T)$.
 The subgroup of $\End(T)$ consisting of all invertible elements of $\End(T)$ is called the \emph{automorphism group} of $T$ and denoted by $\Aut(T)$.
 The  monoid $\End(T)$  acts  on $\partial T$ by setting $f \xi \coloneqq \partial f(\xi)$ for all $f \in \End(T)$ and $\xi \in \partial T$.
 This action is  $1$-Lipschitz with respect to the metric $d$ and therefore equicontinuous.
 Thus, we deduce from Proposition~\ref{p:not-exp-if-infinite} the following result.

\begin{proposition}
\label{p:exp-on-boundary-tree}
Let $T = (V_n,\pi_n)_{n \in \N}$ be a rooted tree and let $S$ be a subsemigroup of $\End(T)$.
Then the action of  $S$ on $\partial T$ is not expansive.
\qed
\end{proposition}

By applying Theorem~\ref{t:equicont-potp}, we obtain the following.

\begin{theorem}
\label{t:boundary-tree}
 Let $T$ be a rooted tree and let $M$ be a finitely generated submonoid of  $\End(T)$.
 Then the action of $M$ on  $\partial T$   has the pseudo-orbit tracing property.
 \qed
\end{theorem}

Given a rooted tree $T = (V_n,\pi_n)_{n \in \N}$, the monoid $\End(T)$ naturally acts on every $V_n$.
This action is given by $fv = f_n(v)$ for all $f = (f_m)_{m \in \N} \in \End(T)$ and $v \in V_n$.
One says that a subsemigroup $S$ of $\End(T)$ is \emph{level-transitive} if the action of $S$ on $V_n$ is transitive for
every $n \in \N$.

\begin{proposition} 
\label{p:semi-nekra}
(cf.~\cite[Proposition~2.4.2]{nekrashevych})
Let $T = (V_n,\pi_n)_{n \in \N}$ be a rooted tree and let $S$ be a subsemigroup of $\End(T)$.
Then the following conditions are equivalent:
\begin{enumerate}[{\rm (a)}]
\item $S$ is level-transitive;
\item the dynamical system $(\partial T, S)$ is minimal.
\end{enumerate}
\end{proposition}

\begin{proof}
Suppose (a). Let $\xi, \eta \in \partial T$ and $\varepsilon > 0$. 
Let $n \in \N$ such that $2^{-n} \leq \varepsilon$ and consider
the vertices $\xi(n), \eta(n) \in V_n$. 
As the action of $S$ on $V_n$ is transitive by (a), there exists  $s   \in S$
such that $s\xi(n) = \eta(n)$. 
By~\eqref{e:distance-frontiere}, this implies  $d(s\xi, \eta) < 2^{-n} \leq \varepsilon$. This shows that the $S$-orbit of $\xi$ is dense in $\partial T$.
The implication (a) $\implies$ (b) follows.
\par
Conversely, suppose (b).
Let $n \in \N$ and $v,w \in V_n$.
Choose rays $\xi,\eta \in \partial T$ such that $v = \xi(n)$ and $w = \eta(n)$.
Since the $S$-orbit of $\xi$ is dense in $\partial T$ by (b),
there exists $s \in S$ such that
$d(s\xi,\eta) < 2^{-n}$.
This implies that  $s\xi(n) = \eta(n)$ and hence $sv = w$.
We deduce that the action of $S$ on $V_n$ is transitive.
This shows that (b) implies (a).
 \end{proof}

As the boundary of every rooted tree is infinite, an immediate consequence of implication (a) $\implies$ (b) in Proposition~\ref{p:semi-nekra}
is the following.

\begin{corollary}
\label{c:level-trans-inf-orbits}
Let $T$ be a rooted tree and let $S$ be a level-transitive subsemigroup of $\End(T)$.
Then every $S$-orbit in $\partial T$ is infinite.
\qed
\end{corollary}

\begin{remark}
Let $S$ and $S'$ be two level-transitive subsemigroups of $\End(T_2)$ (e.g.,~$S = S' = \End(T_2)$). 
We can embed the product semigroup $S \times S'$ into $\End(T_2)$ by setting $(s,s')\vert_{A^0} \coloneqq \Id_{A^0}$ and
\[
(s,s')(a_1, \ldots, a_n) \coloneqq \begin{cases}
(a_1, s(a_2, \ldots,a_n)) & \mbox{ if } a_1 = 0\\
(a_1, s'(a_2, \ldots,a_n)) & \mbox{ if } a_1 = 1
\end{cases}
\]
for all $(s,s') \in S \times S'$ and $(a_1, \ldots, a_n) \in A^{n}$, for all $n \geq 1$.
It is clear that the semigroup $(S \times S') \subset \End(T_2)$ is not level-transitive. 
However, the $(S \times S')$-orbit of every point $\xi \in \partial T_2$ is infinite.
This shows that the converse of the implication stated in  Corollary~\ref{c:level-trans-inf-orbits} is false in general.
\end{remark}

\begin{proposition}
\label{p:smg-level-trans}
(cf.~\cite[Lemma~2.1]{vorobets-free-2007})
Let $T = (V_n,\pi_n)_{n \in \N}$ be a rooted tree.
Let $\Sigma \subset \Aut(T)$ and denote by $S$ (resp.~$M$, resp.~$G$)
the subsemigroup (resp.~submonoid, resp.~subgroup) 
of $\Aut(T)$  generated by $\Sigma$.
Then the following conditions are equivalent:
\begin{enumerate}[\rm (a)]
\item 
$S$ is level-transitive;
\item
 $M$ is level-transitive;
\item 
$G$ is level-transitive.
\end{enumerate}
\end{proposition}

\begin{proof}
The implications (a) $\implies$ (b) and  (b) $\implies$ (c) are trivial since $S \subset M \subset G$.
Let us show the  implication (c) $\implies$ (a).
Suppose that $G$ is level-transitive. 
Fix $n \geq 1$ and let $u, v \in V_n$.
By our hypothesis, we can find $g \in G$ satisfying $g(u) = v$.
As $\Sigma$ generates $G$ as a group, 
there exist an integer $\ell \geq 1$, $\sigma_1, \ldots, \sigma_\ell \in \Sigma$, and $\epsilon_1, \ldots, \epsilon_\ell \in \{-1,1\}$ such that $g = \sigma_1^{\epsilon_1} \cdots \sigma_\ell^{\epsilon_\ell}$.
Setting $N \coloneqq |\Sym(V_n)| = |V_n|!$, we have
$\pi^N = \Id_{V_n}$ for all $\pi \in \Sym(V_n)$.
Consequently, the element $s \coloneqq \sigma_1^{N+\epsilon_1} \cdots \sigma_\ell^{N+\epsilon_\ell} \in S$ satisfies
$s(w) = g(w)$ for all $w \in V_n$.
In particular, we have  $s(u) = g(u) = v$.
This shows that $S$ is level-transitive.
\end{proof}

\begin{example}
\label{ex:automaton-semigroups}
Automaton semigroups (resp.~monoids, resp.~groups), i.e., semigroups (resp.~monoids, resp.~groups)
generated by Mealy machines, yield examples of finitely generated subsemigroups (resp.~submonoids, resp.~subgroups) of $\End(T_k)$.
\par
A \emph{Mealy machine} is a quadruple ${\mathcal M} = (Q, A,\tau, \sigma)$, where
\begin{itemize}
\item $Q$ is a finite set;
\item $A \coloneqq \{0,1, \ldots, k-1\}$ for somme integer $k \geq 2$;
\item $\tau \colon Q \times A \to Q$ is a map; 
\item $\sigma \colon Q \to \Map(A)$ is a map, also written $q \mapsto \sigma_q$.
\end{itemize}
The sets $Q$ and $A$ are respectively called the \emph{set of states} and the \emph{alphabet}, while  $\tau$ and $\sigma$ are respectively called the \emph{transition map} and the \emph{output map} of the machine $\MM$.
Given a Mealy machine ${\mathcal M} = (Q, A,\tau, \sigma)$, one recursively defines
$F_q = F_q^{\mathcal M} = (f_n)_{n \in \N} \in \End(T_k)$, $q \in Q$, by setting $f_0 \coloneqq \Id_{A^0}$ and
\[
f_{n}(a_1,a_2, \ldots, a_n) \coloneqq (\sigma_q(a_1), f'_{n-1}(a_2,\ldots,a_{n})),
\]
where $(f'_n)_{n \in \N} \coloneqq F_{q'} \in \End(T_k)$ and $q' \coloneqq \tau(q,a_1) \in Q$,
for all $(a_1,\ldots,a_n) \in A^n$ and $n \geq 1$.
Then the semigroup $S({\mathcal M}) \subset \End(T_k)$ (resp.\ monoid $M({\mathcal M}) \subset \End(T_k)$) generated by the maps
$F_q^{\mathcal M}$, $q \in Q$, is called the \emph{automaton semigroup} (resp.\ \emph{automaton monoid}) generated by ${\mathcal M}$.
\par
One says that the Mealy machine $\MM$ is \emph{invertible} if  the output map satisfies $\sigma_q \in \Sym(A)$ for all $q \in Q$.
If $\MM$ is invertible then  $F_q = F_q^{\mathcal M} \in \Aut(T_k)$ for all $q \in Q$, and the group $G({\mathcal M}) \subset \Aut(T_k)$
they generate is called the \emph{automaton group} generated by ${\mathcal M}$ (see e.g.~\cite{nekrashevych}, \cite[Exercise~6.60]{csc-ecag}).
We refer the reader to~\cite{nekrashevych} and the references therein for a list of important examples of automaton groups.
The list includes in particular the \emph{Grigorchuk group}~\cite[Example~2.4.65]{nekrashevych},
the \emph{Basilica group} \cite[Example~5.2.20]{nekrashevych},
and the \emph{Hanoi Towers groups} $H(k)$ \cite[Example~2.4.61]{nekrashevych}.  
\par
Proposition~\ref{p:exp-on-boundary-tree} and Theorem~\ref{t:boundary-tree} apply to all automaton monoids and all automaton groups
showing that their action on their tree boundary is not expansive and has the pseudo-orbit tracing property.
\end{example}

\begin{theorem}
\label{t:end-arbre-non-topstab}
Let $k \geq 2$ be an integer  and let $S$ be a subsemigroup (resp.~a submonoid)  of $\End(T_k)$.
Suppose that every orbit of the action of $S$ on $\partial T_k$ is infinite (e.g., $S$ is level-transitive).
Then the dynamical system $(\partial T_k,S)$ is not topologically semistable
(resp.~not topologically stable).
\end{theorem}

\begin{proof}
Denote by $\alpha \in \CSA_S(\partial T_k)$ (resp.~$\alpha \in \CMA_S(\partial T_k)$) the action of $S$ on $\partial T_k$.
We shall show the following strong negation of topological semistability (resp.~topological stability): 
for every neighborhood $N$ of $\alpha$ in $\CSA_S(\partial T_k)$ (resp.~$\CMA_S(\partial T_k)$),
there exists $\beta \in N$ such that there is no map $h \colon \partial T_k \to \partial T_k$ satisfying
\begin{equation}
\label{e:alpha-f-beta}
\alpha_s \circ h = h \circ \beta_s \quad \text{for all } s \in S.
\end{equation}
Suppose that  $N$ is a neighborhood of $\alpha$ in $\CSA_S(\partial T_k)$ (resp.~$\CMA_S(\partial T_k)$). 
Then there exist $n_0 \in \N$ and a finite subset $K \subset S$ such that
\[
\{\gamma \in \CSA_S(\partial T_k)  \text{ (resp.~$\CMA_S(\partial T_k)$)}: \gamma_s(\xi)(n_0) = \alpha_s(\xi)(n_0) \text{ for all } s \in K \text{ and } \xi \in \partial T_k  \} \subset N.
\]
For $f = (f_n)_{n \in \N} \in \End(T_k)$, consider the sequence  $\overline{f} \coloneqq  (\overline{f}_n)_{n \in \N}$,
where $\overline{f}_n \colon A^{n} \to A^n$ is defined by 
$\overline{f}_n \coloneqq f_n$ if $n \leq n_0$, and
\[
\overline{f}_n(a_1,\ldots,a_{n_0},a_{n_0+1}, \ldots, a_n) \coloneqq (f_{n_0}(a_1,\ldots,a_{n_0}), a_{n_0+1}, \ldots, a_n)
\]
for all $(a_1,\ldots,a_n) \in A^n$ if $n > n_0$.
\par
For $n < n_0$, we have $\pi_n \circ \overline{f}_{n+1}= \pi_n \circ f_{n+1} = f_n \circ \pi_n = \overline{f}_{n} \circ \pi_n$.
For $n \geq n_0$, we have
\[
\begin{split}
\pi_n(\overline{f}_{n+1}(a_1,\ldots,a_{n},a_{n+1}))
& = \pi_n(f_{n_0}(a_1, \ldots,a_{n_0}),a_{n_0+1}, \ldots, a_n, a_{n+1})\\
& = (f_{n_0}(a_1,\ldots,a_{n_0}),a_{n_0+1}, \ldots, a_n)\\
& = \overline{f}_{n}(a_1,\ldots,a_n)\\
& = \overline{f}_{n}(\pi_n(a_1,\ldots,a_{n+1})).
\end{split}
\]
 We deduce that $\pi_n \circ \overline{f}_{n+1} = \overline{f}_n \circ \pi_n$ for all $n \in \N$.
This shows that $\overline{f} \in \End(T_k)$.
\par
Consider now the map   $\rho \colon \End(T_k) \to \End(T_k)$ defined by 
$\rho(f) \coloneqq \overline{f}$ for all $f \in \End(T_k)$.
We claim that  $\rho$ is a monoid endomorphism of  $\End(T_k)$.
First observe that
$\rho(1_{\End(T_k)}) =
\rho((\Id_{V_n})_{n \in \N}) =
(\overline{\Id_{V_n}})_{n \in \N} =
(\Id_{V_n})_{n \in \N} = 1_{\End(T_k)}$.
Also, let $f = (f_n)_{n \in \N}, g = (g_n)_{n \in \N} \in \End(T_k)$.
For $n \leq n_0$, we have $\overline{(fg)}_n = (fg)_n = f_ng_n = \overline{f}_n\overline{g}_n$.
For $n > n_0$, we have
\[
\begin{split}
\overline{(fg)}_n(a_1,\ldots,a_{n_0},a_{n_0+1}, \ldots, a_n) & = ((fg)_{n_0}(a_1,\ldots,a_{n_0}), a_{n_0+1}, \ldots, a_n)\\
& = (f_{n_0}(g_{n_0}(a_1,\ldots,a_{n_0})), a_{n_0+1}, \ldots, a_n)\\
& = \overline{f}_{n_0}((g_{n_0}(a_1,\ldots,a_{n_0})), a_{n_0+1}, \ldots, a_n)\\
& = \overline{f}_{n_0}(\overline{g}_{n_0}(a_1,\ldots,a_{n_0}, a_{n_0+1}, \ldots, a_n)).
\end{split}
\]
This shows that $\rho(fg) = \overline{fg} = \overline{f} \cdot \overline{g} = \rho(f)\rho(g)$, and the claim follows.
\par
Observe that the image by $\rho$ of  $f \in \End(T_k)$  is entirely determined by $f_{n_0}$.
It follows that $\rho(\End(T_k))$ is finite with cardinality
\[
|\rho(\End(T_k))| \leq |A^{n_0}|^{|A^{n_0}|} = k^{n_0 k^{n_0}}.
\] 
\par
For $s \in S$, set 
$\beta_s(\xi) \coloneqq \rho(s) \xi = \partial \rho(s)(\xi)$ for all $\xi \in \partial T_k$.
As $\rho$ is a monoid endomorphism, this defines a continuous semigroup (resp.~monoid) action $\beta$ of $S$ on $\partial T_k$.
Moreover, by construction, we have that $\beta \in N$.
\par
Let $h \colon \partial T_k \to \partial T_k$ be a map satisfying~\eqref{e:alpha-f-beta} and let $\xi \in \partial T_k$.
Observe that its $\beta$-orbit $\{\beta_s(\xi) = \rho(s) \xi: s \in S\} \subset \partial T_k$ is finite since $\rho(\End(T_k))$ is finite.
 By~\eqref{e:alpha-f-beta}, this set equals the $\alpha$-orbit of $h(\xi)$, contradicting our assumptions.
This shows that the action of $S$ on $\partial T_k$ is not topologically semistable
(resp.~not topologically stable).
\end{proof}

\begin{example}
\label{ex:grigorchuk}
The class of \emph{branch groups}, which includes the Grigorchuk group and the Hanoi Towers groups, and, more generally, the class of \emph{weakly branch groups}, which contains the Basilica group, 
\cite[Section~2.4.9]{nekrashevych} yield examples of finitely generated level-transitive subgroups of $\Aut(T_k)$. 
It follows from Proposition~\ref{p:smg-level-trans} that if $\Sigma$ is a generating subset of a weakly branch group $G \subset \Aut(T_k)$ then  the semigroup $S$ (resp.~monoid $M$) generated by $\Sigma$ is a level-transitive subsemigroup (resp.~submonoid) of $\End(T_k)$.
We deduce from Theorem~\ref{t:end-arbre-non-topstab} that the corresponding dynamical system $(\partial T_k,S)$  (resp.~ $(\partial T_k,M)$, resp.~ $(\partial T_k,G)$) is not topologically semistable (resp.~not topologically stable).
\end{example}

\def\cprime{$'$}


\begin{thebibliography}{10}

\bibitem{aoki-hidaire-book}
{\sc N.~Aoki and K.~Hiraide}, {\em Topological theory of dynamical systems},
  vol.~52 of North-Holland Mathematical Library, North-Holland Publishing Co.,
  Amsterdam, 1994.
\newblock Recent advances.

\bibitem{bourbaki-top-gen}
{\sc N.~Bourbaki}, {\em \'{E}l\'ements de math\'ematique. {T}opologie
  g\'en\'erale. {C}hapitres 1 \`a 4}, Hermann, Paris, 1971.

\bibitem{bowen-omega-limit-sets}
{\sc R.~Bowen}, {\em {$\omega $}-limit sets for axiom {${\rm A}$}
  diffeomorphisms}, J. Differential Equations, 18 (1975), pp.~333--339.

\bibitem{bowen-lns-1978}
\leavevmode\vrule height 2pt depth -1.6pt width 23pt, {\em On {A}xiom {A}
  diffeomorphisms}, American Mathematical Society, Providence, R.I., 1978.
\newblock Regional Conference Series in Mathematics, No. 35.

\bibitem{brin-stuck}
{\sc M.~Brin and G.~Stuck}, {\em Introduction to dynamical systems}, Cambridge
  University Press, Cambridge, 2002.
	
\bibitem{bucki}
{\sc D.~Bucki}, {\em On the stability and shadowing of tree-shifts of finite type}, Proc. Amer. Math. Soc., {152} 
(2024), pp.~3509--3520.

\bibitem{csc-ecag}
{\sc T.~Ceccherini-Silberstein and M.~Coornaert}, {\em Exercises in cellular
  automata and groups}, Springer Monographs in Mathematics, Springer, Cham,
  [2023] \copyright 2023.
\newblock With a foreword by Rostislav I. Grigorchuk.

\bibitem{chung-lee}
{\sc N.-P. Chung and K.~Lee}, {\em Topological stability and pseudo-orbit
  tracing property of group actions}, Proc. Amer. Math. Soc., 146 (2018),
  pp.~1047--1057.

\bibitem{clifford-preston}
{\sc A.~H. Clifford and G.~B. Preston}, {\em The algebraic theory of
  semigroups. {V}ol. {I}}, vol.~No. 7 of Mathematical Surveys, American
  Mathematical Society, Providence, RI, 1961.

\bibitem{coo-book}
{\sc M.~Coornaert}, {\em Topological dimension and dynamical systems},
  Universitext, Springer, Cham, 2015.
\newblock Translated and revised from the 2005 French original.

\bibitem{fletcher-sawyer}
{\sc P.~Fletcher and J.~Sawyer}, {\em Incompressible topological spaces},
  Glasnik Mat. Ser. III, 4(24) (1969), pp.~299--302.

\bibitem{kelley}
{\sc J.~L. Kelley}, {\em General topology}, Springer-Verlag, New York-Berlin,
  1975.
\newblock Reprint of the 1955 edition [Van Nostrand, Toronto, Ont.], Graduate
  Texts in Mathematics, No. 27.

\bibitem{meyerovitch-2019}
{\sc T.~Meyerovitch}, {\em Pseudo-orbit tracing and algebraic actions of
  countable amenable groups}, Ergodic Theory Dynam. Systems, 39 (2019),
  pp.~2570--2591.

\bibitem{nekrashevych}
{\sc V.~Nekrashevych}, {\em Groups and topological dynamics}, vol.~223 of
  Graduate Studies in Mathematics, American Mathematical Society, Providence,
  RI, [2022] \copyright 2022.

\bibitem{oprocha-2008}
{\sc P.~Oprocha}, {\em Shadowing in multi-dimensional shift spaces}, Colloq.
  Math., 110 (2008), pp.~451--460.

\bibitem{osipov-2013}
{\sc A.~V. Osipov and S.~B. Tikhomirov}, {\em Shadowing for actions of some
  finitely generated groups}, Dyn. Syst., 29 (2014), pp.~337--351.

\bibitem{vorobets-free-2007}
{\sc M.~Vorobets and Y.~Vorobets}, {\em On a free group of transformations
  defined by an automaton}, Geom. Dedicata, 124 (2007), pp.~237--249.

\bibitem{walters-anosov}
{\sc P.~Walters}, {\em Anosov diffeomorphisms are topologically stable},
  Topology, 9 (1970), pp.~71--78.

\bibitem{walters-potp-1978}
{\sc P.~Walters}, {\em On the pseudo-orbit tracing property and its
  relationship to stability}, in The structure of attractors in dynamical
  systems ({P}roc. {C}onf., {N}orth {D}akota {S}tate {U}niv., {F}argo,
  {N}.{D}., 1977), vol.~668 of Lecture Notes in Math., Springer, Berlin-New
  York, 1978, pp.~231--244.

\end{thebibliography}
\end{document}